\documentclass{article}

\usepackage{amsmath,mathrsfs,amsthm,bbm}
\usepackage{txfonts}
\usepackage[all]{xy}

\DeclareMathOperator{\End}{\mathrm{End}}
\DeclareMathOperator{\Aut}{\mathrm{Aut}}

\DeclareMathOperator{\Sp}{\mathrm{Sp}}
\DeclareMathOperator{\SO}{\mathrm{SO}}
\DeclareMathOperator{\GL}{\mathrm{GL}}
\DeclareMathOperator{\SL}{\mathrm{SL}}
\DeclareMathOperator{\tr}{\mathrm{tr}}
\DeclareMathOperator{\OG}{\mathrm{OG}}
\DeclareMathOperator{\IG}{\mathrm{IG}}
\newtheorem{theorem}{Theorem}
\newtheorem{lemma}[theorem]{Lemma}
\newtheorem{proposition}[theorem]{Proposition}
\newtheorem{corollary}[theorem]{Corollary}

\theoremstyle{definition}
\newtheorem{definition}[theorem]{Definition}
\newtheorem{example}[theorem]{Example}
\newtheorem{remark}[theorem]{Remark}
\DeclareMathOperator{\poly}{\mathrm{poly}}
\DeclareMathOperator{\Rep}{\mathrm{Rep}}
\DeclareMathOperator{\Gr}{\mathrm{Gr}}
\DeclareMathOperator{\diag}{\mathrm{diag}}
\DeclareMathOperator{\Spin}{\mathrm{Spin}}
\DeclareMathOperator{\sym}{\mathrm{sym}}
\DeclareMathOperator{\Ad}{\mathrm{Ad}}
 \newcommand{\C}{\mathbb{C}}
 \newcommand{\Z}{\mathbb{Z}}
\newcommand{\cl}{\mathcal{L}}

\newcommand{\frt}{\mathfrak{t}}
\def\Bbb C{{\mathbb C}}
\begin{document}

\title{Representation ring of Levi subgroups versus cohomology ring of flag varieties}

\author{Shrawan Kumar}

\maketitle
\section{Introduction}
Let us begin by recalling the classical result that the cup product structure constants for the singular cohomology with integral coefficients $H^*$ of the Grassmannian of $r$-planes coincide
with the Littlewood-Richardson tensor product structure constants for $\GL_r$. Specifically, the result asserts that there is a $\mathbb{Z}$-algebra 
homomorphism 
$\phi: \Rep _{\poly}(\GL_r) \to H^*(\Gr(r, n)),$
where $\Gr(r, n)$ denotes the Grassmannian of $r$-planes in $\mathbb{C}^n$,  $\Rep_{\poly} (\GL_r)$ denotes the polynomial representation ring of 
$\GL_r$ and $\phi$ takes the irreducible polynomial representation $V(\lambda)$ of  $\GL_r$ corresponding to the partition 
$\lambda:\lambda_1 \geq \dots\geq \lambda_r\geq 0$ to the Schubert class $\epsilon_{v_A(\lambda)}$ corresponding to the same partition 
$\lambda$ if $\lambda_1 \leq n-r$, where $v_A(\lambda)$ is defined in Section 5. If  $\lambda_1 > n-r$, then $\phi (V(\lambda))=0$. 

{\it This work seeks to achieve one  possible generalization  of this classical result for  $\GL_r$ and the Grassmannian 
$\Gr(r,n)$ to the Levi subgroups of any reductive group $G$ and the corresponding  flag varieties.} We also refer to the work  of Tamvakis [T] for another generalization.

Let $G$ be a connected reductive group  over $\mathbb{C}$ with a Borel subgroup $B$ and maximal torus $T\subset B$. Let $P$ be a standard parabolic subgroup with the Levi subgroup $L$ containing $T$. Let $W$ (resp. $W_L$) be the Weyl group of $G$ (resp. $L$). Let $V(\lambda)$ be an irreducible almost faithful representation of $G$ with highest weight $\lambda$ (i.e., the corresponding map $\rho_\lambda: G \to \Aut (V(\lambda))$ has finite kernel). Then, 
Springer defined an adjoint-equivariant regular map with Zariski dense image
$
\theta_{\lambda}:G\to \mathfrak{g}$ (depending upon $\lambda$)
 (cf. Section 3). 
By Lemma \ref{lemma1}, $\theta_\lambda$ takes the maximal torus $T$ to its Lie algebra $\frt$. This induces  a $\C$-algebra homomorphism
$({\theta_\lambda}_{|T})^*: \C[\frt] \to \C[T]$  on the corresponding affine coordinate rings. Since $
\theta_{\lambda}$ is equivariant under the adjoint actions, $({\theta_\lambda}_{|T})^*$ takes $\C[\frt]^{W_L}=S(\frt^*)^{W_L}$ to $\C[T]^{W_L}$. Moreover,  $({\theta_\lambda}_{|T})^*$ 
is injective. Let $\Rep^\C(L)$ be the complexified representation ring of the Levi subgroup $L$. As it  is well known,
$$  \Rep^\C(L) \simeq \C[T]^{W_L}$$
induced from the restriction of the character to $T$. We call the image of  $\C[\frt]^{W_L}$ under  $({\theta_\lambda}_{|T})^*$, the {\em $\lambda$-polynomial subring}  $\Rep^\C_{\lambda-\poly}(L)$ of $\Rep^\C(L)$ (cf. Definition \ref{maindefi}). 

For $G=\GL_n$ and $V(\lambda)$ the defining representation $\C^n$, the ring  $\Rep_{\lambda-\poly}(G):=  \Rep^\C_{\lambda-\poly}(G)\cap  \Rep (G)$
coincides with the standard notion of polynomial representation ring of $\GL_n$ (cf. Section 5). 

 The Borel homomorphism 
$\beta: S(\frt^*) \to H^*(G/B, \C)$ (which is surjective) from the symmetric algebra of $\frt^*$ restricticted  to the $W_L$-invariants gives a surjective 
$\C$-algebra homomorphism $\beta^P:  S(\frt^*)^{W_L} \to H^*(G/P, \C)$. Thus, we get a surjective $\C$-algebra homomorphism 
$\xi^P_\lambda:  \Rep^\C_{\lambda-\poly}(L) \to H^*(G/P, \C)$ (cf. Theorem \ref{thmmain}), which is our main  result. 

Specializing the above result to the case when $G=\GL_n$, $\lambda$ is the first fundamental weight (so that $V(\lambda)$ is the standard defining representation $\C^n$)  and $P=P_r$ (for any $1\leq r \leq n-1$) is the maximal parabolic subgroup so that the flag variety 
$G/P_r$ is the Grassmannian $\Gr(r,n)$, we recover the above classical result as shown in Section 5 (cf. Theorem \ref{thm3}). 

We determine  the $\lambda$-polynomial representation ring $\Rep^\C_{\lambda-\poly}(G)$, for $\lambda$ the first fundamental weight $\omega_1$
(i.e., $V(\lambda)$ is the defining representation) of the classical groups:
$\SO_n, \Sp_{2n}$ in Section 6 (cf. Proposition \ref{prop5}). In this case, the Springer morphism coincides with the classical Cayley transform. Recall that the defining representations of the classical groups have minimum Dynkin index (cf. [KN, $\S$4]). We believe that for the exceptional groups as well, the irreducible representation $V(\lambda)$ with minimum Dynkin index might be most `appropriate' to consider the Springer morphism. Recall (loc. cit.) that 
for the exceptional groups: $G_2, F_4, E_6, E_7, E_8$, the representation $V(\lambda)$ has minimum Dynkin index for 
$\lambda= \omega_1, \omega_4, \omega_1$ (and $\omega_6$), $ \omega_7, \omega_8$ respectively.  

We  partially determine  the homomorphism $\xi^P_\lambda:  \Rep^\C_{\lambda-\poly}(L) \to H^*(G/P, \C)$ (with respect to the defining representation: $\lambda =\omega_1$)  for all the maximal parabolic subgroups $P$ in the classical groups $\Sp_{2n}, \SO_{2n+1}$ and $\SO_{2n}$ in Sections 7, 8 and 9  (cf. Propositions \ref{prop8}, \ref{prop9} and \ref{prop10} respectively). 

We determine the homomorphism $\xi^B_{\omega_1}:  \Rep^\C_{\omega_1-\poly}(T) \to H^*(G/B, \C)$  for the Borel subgroups $B$ in the classical groups
$\Sp_{2n}, \SO_{2n+1}$ and $\SO_{2n}$  in Section 10 (cf. Proposition \ref{lastpropo}). (In this case, $T$ is of course the Levi subgroup of $B$.)

\vskip2ex

\noindent
{\bf Acknowledgements:} I am grateful to Michele Vergne whose question led me to this work. This work was partially done during my visit to the University of Sydney, hospitality of which is gratefully acknowledged. This work was partially supported by the NSF grant DMS- 1501094.

\section{Notation}

Let $G$ be a connected reductive group  over $\mathbb{C}$ with a Borel subgroup $B$ and maximal torus $T\subset B$. Let $P$ be a standard parabolic subgroup with the Levi subgroup $L$ containing $T$. We denote their Lie algebras by the corresponding Gothic characters: $\mathfrak{g}, \mathfrak{b}, \mathfrak{t},\mathfrak{p}$
respectively. We denote by $\Delta=\{\alpha_1, \dots, \alpha_\ell\}\subset \mathfrak{t}^*$ the set of simple roots. The fundamental weights of $\mathfrak{g}$ are denoted by  $\{\omega_1, \dots, \omega_\ell\}\subset \mathfrak{t}^*$. Let $W$ (resp. $W_L$) be the Weyl group of $G$ (resp. $L$).
Then, $W$ is generated by the simple reflections $\{s_i\}_{1\leq i \leq \ell}$. Let $W^P$ denote the set of smallest coset representatives in the cosets in $W/W_L$. {\it Throughout the paper we follow the indexing convention as in [Bo, Planche
I - IX].} 

Let $X(T)$ be the group of characters of $T$ and let $D\subset X(T)$ be the set of dominant characters  (with respect to the given choice of $B$ and hence positive roots, which are the roots of $\mathfrak{b}$). Then, the isomorphism classes of  finite dimensional irreducible representations of $G$ are bijectively parameterized by $D$ under the correspondence $\lambda \in D \leadsto V(\lambda)$, where
$V(\lambda)$ is the irreducible representation of $G$ with highest weight $\lambda$. We call $V(\lambda)$ {\it almost faithful} if the corresponding map $\rho_\lambda: G \to \Aut (V(\lambda))$ has finite kernel. 

Recall the 
Bruhat decomposition for the flag variety:
$$G/P=\sqcup_{w\in W^P}\, \Lambda_w^P,\,\,\,\text{where}\,\, \Lambda^P_w:= BwP/P.$$
Let $\bar{\Lambda}_w^P$ denote the closure of $\Lambda_w^P$ in $G/P$.  We denote by 
$[\bar{\Lambda}_w^P] \in H_{2 \ell(w)}(G/P, \mathbb{Z})$ its
fundamental class. Let $\{\epsilon^P_w\}_{w\in W^P}$ denote the Kronecker dual basis of the cohomology, i.e., 
$$\epsilon^P_w([\bar{\Lambda}_v^P])= \delta_{w,v}, \,\,\,\text{for any}\,\, v,w\in W^P.$$
Thus, $\epsilon^P_w$ belongs to the singular cohomology:
$$\epsilon^P_w\in H^{2 \ell(w)}(G/P, \mathbb{Z}).$$

\section{Springer Morphism}
\begin{definition}\label{def1}
Let $V(\lambda)$ be any almost faithful irreducible representation of $G$. Following Springer (cf. [BR, $\S$9]), define the map
$$
\theta_{\lambda}:G\to \mathfrak{g}\quad \text{(depending upon $\lambda$)}
$$
as follows:
\[
\xymatrix{
G\ar[r]^-{\rho_{\lambda}}\ar[dr]^{\theta_{\lambda}} & \Aut (V(\lambda))\subset \End (V(\lambda))=\mathfrak{g}\oplus \mathfrak{g}^{\perp}\ar[d]^{\pi}\\
 & \mathfrak{g}
}
\]
where  $\mathfrak{g}$ sits canonically inside $\End(V(\lambda))$ via the derivative $d\rho_{\lambda}$, the orthogonal complement $\mathfrak{g}^{\perp}$
is  taken with respect to the standard conjugate $ \Aut(V(\lambda))$-invariant form on $\End (V(\lambda))$: $\langle A, B\rangle :=\tr (AB)$,  and $\pi$ is the projection to the $\mathfrak{g}$-factor. (By considering a compact form $K$ of $G$, it is easy to see that  $ \mathfrak{g} \cap \mathfrak{g}^{\perp} =\{0\}$.)

Since $\pi\circ d\rho_{\lambda}$ is the identity map, $\theta_{\lambda}$ is a local diffeomorphism at $1$ (and hence with Zariski dense image). Of course, by construction, $\theta_{\lambda}$ is an algebraic morphism. Moreover, since the decomposition $ \End (V(\lambda))=\mathfrak{g}\oplus \mathfrak{g}^{\perp}$ is $G$-stable, it is easy to see that $\theta_\lambda$ is $G$-equivariant under conjugation.
\end{definition}
\begin{lemma} \label{lemma1} The above morphism restricts to ${\theta_\lambda}_{|T}: T \to \mathfrak{t}$.
\end{lemma}
\begin{proof} Take any $t\in T$ and write 
$$\theta_\lambda (t)= h+\sum_{\alpha\in R}\,x_\alpha,\,\,\,\text{for}\,\, h\in  \mathfrak{t}, \,\,\text{and}\,\, x_\alpha \in  \mathfrak{g}_\alpha,$$
where $R$ is the set of all the roots and $ \mathfrak{g}_\alpha$ is the root space of  $\mathfrak{g}$ corresponding to the root $\alpha$.  Now, since $\theta_\lambda$ is conjugation invariant, we get that 
$$\theta_\lambda (t)= h+\sum_{\alpha\in R}\,(\Ad s)\cdot x_\alpha,\,\,\,\text{for any }\,\, s\in  T.$$
From this we see that each $ x_\alpha =0$, i.e., $\theta_\lambda(t) \in \mathfrak{t}$, proving the lemma.
\end{proof}
\begin{example} \label{newexample1} The Springer morphism $\theta_\lambda: G \to \mathfrak{g}$, in general, indeed depends upon the choice of $\lambda$. For example, the Springer morphism $\theta_{\omega_1}: \SL_2 \to \mathfrak{s}l_2$ restricted to the diagonal torus can easily seen to be
$$
\theta_{\omega_1}
\left(
\begin{matrix}
z & 0\\
0 & z^{-1}
\end{matrix}
\right) = \left(
\begin{matrix}
\frac{z-z^{-1}}{2} & 0\\
0 & - \frac{z-z^{-1}}{2}
\end{matrix}
\right).
$$
On the other hand,  the Springer morphism $\theta_{2\omega_1}: \SL_2 \to \mathfrak{s}l_2$ restricted to the diagonal torus
is given by
$$
\theta_{2\omega_1}
\left(
\begin{matrix}
z & 0\\
0 & z^{-1}
\end{matrix}
\right) = \left(
\begin{matrix}
\frac{z^2-z^{-2}}{2} & 0\\
0 & - \frac{z^2-z^{-2}}{2}
\end{matrix}
\right).
$$
\end{example}

\section{Main Result}

We follow the notation and assumptions as in Section 2. In particular, $G$ is a connected reductive group and $P$ a standard parabolic subgroup with Levi subgroup  $L$ containing the chosen maximal torus $T$. Take an almost faithful irreducible $G$-module 
$V(\lambda)$. As in Section 3, this gives rise to the Springer morphism $\theta_\lambda: G \to \mathfrak{g}$. 
which restricts to
$
{\theta_\lambda}_{|T}:T\to \mathfrak{t}
$ (cf. Lemma \ref{lemma1}).  

 For any $\mu \in X(T)$, we have a  $G$-equivariant
 line bundle $\cl (\mu)$  on $G/B$ associated to the principal $B$-bundle $G\to G/B$
via the one dimensional $B$-module $\mu^{-1}$. (Any  $\mu \in X(T)$ extends
 uniquely to a character of $B$.) The one dimensional $B$-module $\mu$ is also denoted by
 $\mathbb{C}_\mu$. Recall the surjective  Borel homomorphism
 $$\beta : S(\mathfrak{t}^*) \to H^*(G/B, \mathbb{C}),$$
which takes a character $ \mu \in X(T)$ to the first Chern class of the  line bundle $\cl(\mu)$. (We realize 
$X(T)$ as a lattice in $\mathfrak{t}^*$ via taking derivative.) We then extend this map linearly over $\mathbb{C}$ to $\mathfrak{t}^*$ and extend further as a graded algebra homomorphism from $ S(\mathfrak{t}^*)$ (doubling the degree). Under the Borel homomorphism, as, e.g., in [Ku, Exercise 11.3.E.1],
\begin{equation}\label{eqnborel} \beta(\omega_i)=\epsilon_{s_i}^B,\,\,\,\text{for any fundamental weight}\,\, \omega_i.
\end{equation}

Fix a compact form $K$ of $G$. In particular,  $T_o:=K\cap T$ is a (compact) maximal torus of $K$. Then, $W\simeq N(T_o)/T_o$, where $N(T_o)$ is the normalizer of $T_o$ in $K$. Recall that $\beta$ is $W$-equivariant under  the standard action of $W$ on  $S(\mathfrak{t}^*)$ and the $W$-action on $H^*(G/B, \mathbb{C})$ induced from the $W$-action on $G/B\simeq K/T_o$ via 
$$(nT_o)\cdot (kT_o):= kn^{-1}T_o,\,\,\,\text{for}\,\, n\in N(T_o)\,\,\,\text{and}\,\,  k\in K.$$
Thus, for any standard parabolic subgroup $P$ with the Levi subgroup $L$ containing $T$, restricting $\beta$, we get a surjective graded algebra homorphism:
$$\beta^P : S(\mathfrak{t}^*)^{W_L} \to H^*(G/B, \mathbb{C})^{W_L}\simeq  H^*(G/P, \mathbb{C}),$$
where the last isomorphism, which is induced from the projection $G/B \to G/P$,  can be found, e.g.,  in [Ku, Corollary 11.3.14]. 

Now, the Springer morphism ${\theta_\lambda}_{|T}:T\to \mathfrak{t}$ (restricted to $T$) gives rise to the corresponding $W$-equivariant injective  algebra homomorphism 
on the affine coordinate rings:
$$ ({\theta_\lambda}_{|T})^*: \mathbb{C}[\mathfrak{t}]=S(\mathfrak{t}^*)\to  \mathbb{C}[T].$$
Thus,  on restriction to $W_L$-invariants, we get an injective algebra homomorphism 
$$ {\theta_\lambda (P)}^*: \mathbb{C}[\mathfrak{t}]^{W_L}=S(\mathfrak{t}^*)^{W_L}\to  \mathbb{C}[T]^{W_L}.$$
Now, let $\Rep (L)$ be the representation ring of $L$ and let  $\Rep^\mathbb{C} (L):= \Rep (L)\otimes_{\mathbb{Z}}\,\mathbb{C}$ be its complexification.  Then, as it is well known,   
\begin{equation} \label{eqn1}\Rep^\mathbb{C} (L)\simeq  \mathbb{C}[T]^{W_L}
\end{equation}
obtained from taking the character of an $L$-module restricted to $T$. 

A representation $V$ of $L$, thought of as an element of $\Rep (L)$, is denoted by $[V]$.  {\it We will often identify a virtual representation of $L$ with its character restricted to $T$ (which is automatically $W_L$-invariant).}
\begin{definition} \label{maindefi} We will call a virtual character $\chi\in  \Rep^\mathbb{C} (L)$ of $L$ a {\it $\lambda$-polynomial character} if the corresponding function in 
$ \mathbb{C}[T]^{W_L}$ is in the image of $ {\theta_\lambda (P)}^*$. The set of all $\lambda$-polynomial characters of $L$, which is, by definition,  a subalgebra of $\Rep^\mathbb{C} (L)$ isomorphic to the algebra $S(\mathfrak{t}^*)^{W_L}$, is denoted by $\Rep^\mathbb{C}_{\lambda-\poly} (L)$. Of course, the map $ {\theta_\lambda (P)}^*$ induces an algebra isomorphism (still denoted by)
$$ {\theta_\lambda (P)}^*: S\left(\mathfrak{t}^*\right)^{W_L}\simeq \Rep^\mathbb{C}_{\lambda-\poly} (L),$$
under the identification \eqref{eqn1}.
\end{definition}

Putting together the above identifications, we get the following  main result of this note.

\begin{theorem} \label{thmmain}  Let $V(\lambda)$ be an almost faithful irreducible $G$-module and let $P$ be any standard parabolic subgroup. Then,  
the above maps (specifically $\beta^P\circ ({\theta_\lambda (P)}^*)^{-1} $) give
 rise to  a surjective $\mathbb{C}$-algebra homomorphism 
$$\xi_\lambda^P: \Rep^\mathbb{C}_{\lambda-\poly}(L)  \to H^*(G/P, \C).$$

Moreover, let $Q$ be another standard  parabolic subgroup with Levi subgroup $R$ containing $T$ such that $P\subset Q$ (and hence 
$L\subset R$). Then, we have the following commutative diagram:\[
\xymatrix{
\Rep^\mathbb{C}_{\lambda-\poly}(R)  \ar[d]^{\gamma}\ar[r]^{\xi_\lambda^Q} & H^*(G/Q, \C)
\ar[d]^{\pi^*}\\
\Rep^\mathbb{C}_{\lambda-\poly}(L)\ar[r]^{\xi_\lambda^P} &H^*(G/P, \C),
}
\]
where $\pi^*$ is induced from the standard projection $\pi:G/P \to G/Q$ and $\gamma$ is induced from the restriction of representations. 
\end{theorem}
\begin{example} The subalgebra $\Rep^\mathbb{C}_{\lambda-\poly}(G) \subset  \Rep^\mathbb{C}(G)$, in general, indeed depends 
upon the choice of $\lambda$. For example, for $G=\SL_2$, following  Example \ref{newexample1}, 
$$ \Rep^\mathbb{C}_{\omega_1-\poly}(\SL_2) = \C[(z-z^{-1})^2],$$
whereas
$$ \Rep^\mathbb{C}_{2\omega_1-\poly}(\SL_2) = \C[(z^2-z^{-2})^2],$$
for the maximal torus in $\SL_2$  given by
$$
T=
\left\{\left(
\begin{matrix}
z & 0\\
0 & z^{-1}
\end{matrix}
\right) :z\in \C^*\right\}.
$$
\end{example}
\section{Specialization of Theorem \ref{thmmain} to the case of $G=\GL_n$}
Theorem \ref{thmmain} specialized to $G=\GL_n$ and $P$  the standard maximal parabolic subgroups gives the following classical result. As mentioned in the Introduction, our  motivation behind this work was to seek a generalization of this classical result for an arbitrary reductive group $G$ and {\it any} parabolic subgroup $P$.

For $G=\GL_n$, we take 
  $B$ to be the (standard)
Borel subgroup consisting of invertible upper triangular matrices 
and $T$ to be the subgroup
 consisting of invertible diagonal matrices. Then,
 \[\frt=\{\bar{\bf t}=\text{diag}(t_1, \dots , t_{n}): t_i\in \mathbb{C}\}. \]
The  simple roots and simple coroots are given respectively by
  \[\alpha_i(\bar{\bf t})=t_i-t_{i+1}\,\,\,\text{and}\,\,  \alpha_i^\vee=\text{diag}(0, \dots, 0,1,-1,0, \dots,
  0),\,\,\,\text{for any}\,\, 1\leq i\leq n-1,\]
 where $1$ is placed in the $i$-th place.
We have the fundamental weights:
\[\omega_i(\bar{\bf t})=t_1+\dots+t_i, \,\,\,\text{for any}\,\, 1\leq i\leq n.\]

  The Weyl group $W$ can be identified with the symmetric group $S_{n}$, which acts
  via the permutation of the coordinates of $\bar{\bf t}$. Let $\{s_1, \dots, s_{n-1}\} \subset S_{n}$
  be the (simple)
  reflections corresponding to the simple roots $\{\alpha_1, \dots, \alpha_{n-1}\}$
  respectively. Then,
  \[s_i=(i,i+1).\]
  For any $1\leq r\leq n-1$, let $P_r\supset B$ be the (standard) maximal parabolic
  subgroup of $\GL_n$ such that its unique Levi subgroup $L_r$ containing $T$
  has for its simple roots $\{\alpha_1, \dots, \hat{\alpha}_r, \dots, \alpha_{n-1}\}$.
  Then, $\GL_n/P_r$ can be identified with the Grassmannian $\Gr(r, n)$ of $r$-dimensional subspaces of $\Bbb C^{n}$. 
Moreover, the set
  of minimal coset representatives $W^{P_r}$ of $W/W_{L_r}$ can be identified
with the set of $r$-tuples
\[S(r,n)=\{A:=1\leq a_1 < \cdots < a_r \leq n \}.\]
Any such $r$-tuple $A$ represents the permutation
\[v_A=(a_1,\dots, a_r,a_{r+1}, \dots, a_{n}),\,\,\, i\mapsto a_i,\]
where $\{a_{r+1}< \cdots < a_{n}\}=[n]\setminus \{a_1, \dots, a_r\}$
and
$ [n]:=\{1, \dots, n\}.$

Recall that an irreducible representation of $\GL_r$ is called a {\it polynomial representation} if its character 
$\GL_r \to \C$ extends as a regular map $\mathfrak{g}l_r \to \C$, where  $\mathfrak{g}l_r$ denotes the space of all $r\times r$-matrices over $\C$. Let $\Rep_{\poly}(\GL_r)$ denote the subring of the representation ring  $\Rep(\GL_r)$ spanned by the (irreducible) polynomial representations of $\GL_r$.
($\Rep_{\poly}(\GL_r)$ is indeed a subring [F, $\S$8.3].)  As in [F, Theorem 2, $\S$8.2],  the (irreducible) polynomial representations of $\GL_r$ are parameterized by the partitions
$\lambda=(\lambda_1 \geq \dots \geq \lambda_r\geq 0)$, where the corresponding irreducible representation $V(\lambda)$  has highest weight 
$\bar{\lambda}$, 
$$\bar{\lambda}({\bf t}) :=t_1^{\lambda_1}\dots  t_r^{\lambda_r}, \,\,\,\text{for any invertible diagonal matrix}\,\, {\bf t}=(t_1, \dots, t_r).$$
Define the $\mathbb{Z}$-linear map 
$$\xi : \Rep_{\poly}(\GL_r) \to H^*(\Gr(r,n), \mathbb{Z}), \,\,[V(\lambda)]\mapsto \epsilon_{v_{A(\lambda)}}^{P_r},\,\,\,\text{if}
\,\, \lambda_1\leq n-r,$$
and $\xi([V(\lambda)]):= 0$, if  $\lambda_1> n-r$, where $A(\lambda):= (1+\lambda_r< \dots < r+\lambda_1)$.
Now, we recall the following classical result (cf. [F, $\S$9.4; Proposition 9 of $\S$10.6;  and the identity (1) of $\S$10.2]).
\begin{theorem} \label{thm2}The above map $\xi$ is a surjective algebra homomorphism.
\end{theorem}
 Take $G=\GL_n$ and $\lambda=\omega_1$  (so that $V(\lambda)$ is the standard representation  $ \C^n$). Then, clearly, 
\begin{equation}\label{eqnnew2}\theta_{\omega_1}: \GL_n \to \mathfrak{g}l_n\,\,\,\text{ is the canonical inclusion}.
\end{equation} Thus, in this case, it is easy to see that 
\begin{equation} \label{eqn2}
 \Rep_{\omega_1-\poly}(\GL_n) = \Rep_{\poly}(\GL_n),
\end{equation}
where $\Rep_{\omega_1-\poly}(\GL_n) :=\Rep^\C_{\omega_1-\poly}(\GL_n)\cap \Rep (\GL_n)$.
For $1 \leq r\leq n-1$, the Levi subgroup $L_r$ of $P_r$ containing $T$ is the subgroup $\GL_r\times \GL_{n-r}$ of $\GL_n$.
From the identity \eqref{eqn2}, it is easy to see that 
\begin{equation} \label{eqn3}
\Rep_{\omega_1-\poly}^\C(L_r) \simeq \left(\Rep_{\poly}(\GL_r) \otimes  \Rep_{\poly}(\GL_{n-r})\right)\otimes_\mathbb{Z} \C.
\end{equation}
Following Theorem \ref{thmmain}, we have the $\C$-algebra homomorphism:
$$\xi^{P_r}_{\omega_1}: \Rep_{\omega_1-\poly}^\C(L_r)\to H^*(\Gr(r,n), \C).$$
Of course, we have a ring homomorphism
$$i:  \Rep_{\poly}(\GL_r) \to   \Rep_{\poly}(\GL_r)\otimes  \Rep_{\poly}(\GL_{n-r})$$
obtained from tensoring a $\GL_r$-module with the trivial one dimensional $\GL_{n-r}$-module. Then, we have the following result:
\begin{theorem} \label{thm3} The $\C$-algebra homomorphism $\xi^{P_r}_{\omega_1}: \Rep_{\omega_1-\poly}^\C(L_r)\to H^*(\Gr(r,n), \C)$
restricted to $ \Rep_{\poly}(\GL_r)$ via $i$ (under the isomorphism \eqref{eqn3}) coincides with the homomorphism $\xi$ of Theorem 
\ref{thm2}.
\end{theorem}
\begin{proof} Since $\xi^{P_r}_{\omega_1}$ and $\xi$ are both algebra homomorphisms and the ring $ \Rep_{\poly}(\GL_r)$ is generated by the fundamental representations $[V(\omega_1)], \dots , [V(\omega_r)]$, it suffices to prove that 
\begin{equation} \label{eqn4}
\xi^{P_r}_{\omega_1}([V(\omega_i)]) = \xi ([V(\omega_i)]), \,\,\,\text{for all}\,\, 1\leq i\leq r.
\end{equation}
Now, by definition,
\begin{equation} \label{eqn5}
 \xi ([V(\omega_i)]) = \epsilon^{P_r}_{s_{r-i+1}\dots s_r}.
\end{equation}
Moreover, from the definition of $\xi^{P_r}_{\omega_1}$, by \eqref{eqnnew2},
\begin{equation} \label{eqn6}
\xi^{P_r}_{\omega_1}([V(\omega_i)]) = \beta \left(e_i(x_1, \dots , x_r)\right), 
\end{equation}
where $x_k \in \mathfrak{t}^*$ is the linear form which takes 
$\bar{\bf t}=(t_1, \dots, t_r, t_{r+1}, \dots, t_n)\in \mathfrak{t}$ to $t_k$,  $e_i$ is the $i$-th elementary symmetric function and $\beta$ is the Borel homomorphism defined in Section 4. Now, by [F, Proposition 8, $\S$10.6], 
\begin{equation} \label{eqn7}
 \beta \left(e_i(x_1,  \dots , x_r)\right)=  \epsilon^{P_r}_{s_{r-i+1}\dots s_r}.\end{equation}
Combining the equations \eqref{eqn5} - \eqref{eqn7}, we get the equation  \eqref{eqn4}. This proves the theorem.

\end{proof}
\begin{remark} Observe that our Theorem \ref{thmmain} in the case of $G=\GL_n$ and $\lambda=\omega_1$ extends the classical 
Theorem \ref{thm2} for $H^*(\Gr(r,n))$ to the cohomology $H^*(\GL_n/B)$ of the full flag variety.
\end{remark}

\section{Determination of $ \Rep_{\omega_1-\poly}(L)$ for other classical groups}

{\it From now on we only consider the Springer morphism for classical groups with respect to the first fundamental weight $\lambda= \omega_1.$ So, we will abbreviate $\theta_{\omega_1}$ by $\theta$ and $\Rep^\C_{\omega_1-\poly}(G)$ by $\Rep^\C_{\poly}(G)$. In this case, $\theta_{\omega_1}$ is the classical Cayley transform.}

We choose quadratic  forms on $\mathbb{C}^{2n}, \mathbb{C}^{2n+1}$ (resp.  alternating form on $\mathbb{C}^{2n}$)  so that $\SO_{2n}$, $\SO_{2n+1}$ (resp. $\Sp_{2n}$) are given respectively by
\begin{align*}
\SO_{2n} &= \{g\in \SL_{2n}:(g^{t})^{-1}=E_{D}gE^{-1}_{D}\}\\
\SO_{2n+1} &= \{g\in \SL_{2n+1}:(g^{t})^{-1}=E_{B}gE^{-1}_{B}\}\\
\Sp_{2n} &= \{g\in \SL_{2n}:(g^{t})^{-1}=E_{C}gE_{C}^{-1}\},
\end{align*}
where $E_{D}$ is the antidiagonal matrix with all its antidiagonal entries $1$; $E_{B}$ is the antidiagonal matrix with all its antidiagonal entries $1$ except the $(n+1,n+1)$-th entry which is $2$; $E_{C}$ is the block matrix 
$$
E_{C}=
\left(
\begin{matrix}
0 & - J_{n}\\
J_{n} & 0
\end{matrix}
\right),
$$
where $J_{n}$ is the antidiagonal $n\times n$ matrix with all its antidiagonal entries $1$. (The suffix $D,B,C$ refers to the types of the corresponding groups.)

Depending upon the case, denote $E_{D}$, $E_{B}$ or $E_{C}$ by the common symbol $E$. We recall the expression of the Springer morphism for these groups with respect to their standard representations $V(\omega_1)$, where $\omega_1$ is the first fundamental weight.

\begin{lemma}\label{lem2}
The Springer morphism $\theta:G\to \mathfrak{g}$\, for $G=\SO_{2n}$, $\SO_{2n+1}$ or $\Sp_{2n}$ is given by 
$$
g\mapsto \frac{g-E^{-1}g^{t}E}{2},\,\,\,\text{for}\,\, g\in G.
$$
(Observe that this is the Cayley transform.)
\end{lemma}

\begin{proof} 
The lemma follows immediately since under the decomposition
$$
\End (V(\omega_1)) = \mathfrak{g}\oplus \mathfrak{g}^{\perp},$$
any $A\in \End (V(\omega_1))$ decomposes as 
$$A = \frac{(A-E^{-1}A^{t}E)}{2}+\frac{(A+E^{-1}A^{t}E)}{2}.$$
\end{proof}

Take the maximal tori in $\Sp_{2n}$, $\SO_{2n}$ and $\SO_{2n+1}$ respectively as follows:
\begin{equation}\label{eqnnew201}
T_{C} = T_{D}=\left\{{\bf t}=
\diag \bigl(t_{1},  \dots, t_n, t_n^{-1}, \dots, 
 t^{-1}_{1}\bigr):t_{i}\in \mathbb{C}^{*}\right\}
\end{equation}
\begin{equation}\label{eqnnew202}
T_{B} = \left\{{\bf t}=
\diag \bigl(
t_{1}, \dots, t_{n}, 1, 
 t^{-1}_{n}, \dots ,  t^{-1}_{1}\bigr)
:~t_{i}\in \mathbb{C}^{*}
\right\}.
\end{equation}
Their Lie algebras are given respectively by 
\begin{equation}\label{eq201}
\frt_{C} = \frt_{D}=\left\{\bar{\bf t}=
\diag \bigl(x_{1},  \dots, x_n, -x_n, \dots, 
 -x_{1}\bigr):x_{i}\in \mathbb{C}\right\}
\end{equation}
\begin{equation} \label{eq202}
\frt_{B} = \left\{\bar{\bf t}=
\diag \bigl(
x_{1}, \dots, x_{n}, 0, 
 -x_{n}, \dots ,  -x_{1}\bigr)
:~x_{i}\in \mathbb{C}
\right\}.
\end{equation}
From the description of the Springer morphism given above, we immediately get the following:

\begin{corollary}\label{coro3}
Restricted to the maximal torus as above, we get the following description of the Springer map $\theta$:

(a) $G=\SO_{2n}:\theta({\bf t})=
\diag \bigl(
\frac{t_{1}- t^{-1}_{1}}{2}, \dots, 
  \frac{t_{n} -t_{n}^{-1}}{2}, 
  -(\frac{t_{n}-t^{-1}_{n}}{2}), \dots, 
  -(\frac{t_{1}-t^{-1}_{1}}{2})\bigr)$

(b) $G=\Sp_{2n}:$ Same as in the above case of $G=\SO_{2n}$.

(c) $G=\SO_{2n+1}:$
$
\theta({\bf t})=
\diag \bigl(
\frac{t_{1} -t^{-1}_{1}}{2}, \dots, 
 \frac{t_{n} -t_{n}^{-1}}{2}, 0, 
  -(\frac{t_{n}-t^{-1}_{n}}{2}), \dots, 
  -(\frac{t_{1}-t^{-1}_{1}}{2})\bigr).$ 
\end{corollary}

The following result follows easily from Corollary~\ref{coro3} together with the description of the Weyl group.

\begin{proposition}\label{prop5}
Let $G=\SO_{2n}$, $\SO_{2n+1}$ or $\Sp_{2n}$ and let   $f:T\to \mathbb{C}$  be a regular map. Then, $f\in \Rep_{\poly}^\C(G)$  if and only if the following is satisfied:

(a) $G=\Sp_{2n}$ or $\SO_{2n+1}$: There exists a {\it symmetric} polynomial $P_f(x_{1},\ldots,x_{n})$ such that 
\begin{equation*}
f({\bf t})=P_f\left((\frac{t_{1}-t^{-1}_{1}}{2})^{2},\dots,(\frac{t_{n}-t^{-1}_{n}}{2})^{2}\right), \,\,\text{for}\,\,{\bf t} \in T_C \text{ given by}\, \eqref{eqnnew201}
\,\text{or}\,\,
 {\bf t} \in T_B  \, \text{ given by}\, \eqref{eqnnew202}.
\end{equation*}

(b) $G=\SO_{2n}$: There exist  symmetric polynomials $P_f(x_{1},\ldots,x_{n})$ and  $Q_f(x_{1},\ldots,x_{n})$ satisfying 
\begin{align*}
f({\bf t})&=P_f\left((\frac{t_{1}-t^{-1}_{1}}{2})^2,\dots, (\frac{t_{n}-t^{-1}_{n}}{2})^2\right) +\\
&\left((\frac{t_{1}-t^{-1}_{1}}{2})\dots (\frac{t_{n}-t^{-1}_{n}}{2})\right) 
Q_f\left((\frac{t_{1}-t^{-1}_{1}}{2})^2,\dots, (\frac{t_{n}-t^{-1}_{n}}{2})^2\right),
 \, \text{for}\,\, {\bf t}\in T_D \,\text{ given by}\, \eqref{eqnnew201}.
\end{align*}
\end{proposition}

We recall the definition of $\lambda$-rings due to Grothendieck  from [AT, $\S$1].
\begin{definition}\label{defi6}
A {\it special $\lambda$-ring} is, by definition, a commutative ring $R$ with identity with a map
$$
 \lambda : R\to R[[q]] ,\,\,\,
x\mapsto \sum\limits_{i\geq 0}\lambda^{i}(x)q^{i},
$$
which satisfies the following:

(1) $\lambda^{0}(x)=1$

(2) $\lambda^{1}(x)=x$, for all $x\in R$

(3) $\lambda(x+y)=\lambda(x)\lambda(y)$, for all $x$, $y\in R$

(4) $\lambda (1)= 1+q$, and

(5) There are universal (independent of $R$) polynomials $P_{k}$ and $P_{k,l}$ over $\mathbb{Z}$ such that
\begin{align*}
& \lambda^{k}(xy)=P_{k}\left(\lambda^1(x),\ldots,\lambda^{k}(x),\lambda^1(y),\ldots,\lambda^{k}(y)\right)\\
\text{and}\quad & \lambda^{k}(\lambda^{l}x)=P_{k,l}\left(\lambda^{1}(x),\ldots,\lambda^{kl}(x)\right),\text{~~ for all~~ } k, l\geq 1.
\end{align*}
\end{definition}
The following example can be found in [FH, Exercise 23.39].
\begin{example}\label{exam7}
For any reductive group $G$, the representation ring $\Rep (G)$ is a special $\lambda$-ring, where for any representation $V$,
$$
\lambda([V]):=\sum\limits_{i\geq 0}\,[\Lambda^{i}(V)] q^{i}
$$
and extend it to all of $\Rep (G)$ by demanding the properly (3) of the definition of special $\lambda$-rings.
\end{example}

\begin{lemma}\label{lem8}
Let $G=\SO_{2n+1}$ or $\Sp_{2n}$. Then,

(a) The subring $\Rep_{\poly}(G)\subset \Rep (G)$ of $\omega_1$-polynomial characters is a special $\lambda$-subring,
where
$$ \Rep_{\poly}(G):= \Rep^\C_{\poly}(G)\cap \Rep (G).$$
(b) Moreover, the character
\begin{equation*}
\chi({\bf t})=\sum\limits^{n}_{i=1}\left(t^{2}_{i}+t^{-2}_{i}\right)\in \Rep_{\poly}(G), \,\,\,\text{for} \,\,  {\bf t}\in T_C  \,\,\text{ given by}\,\, \eqref{eqnnew201}
\,\,\text{or}\,\,
 {\bf t} \in T_B  \,\,\text{ given by}\,\, \eqref{eqnnew202} 
\end{equation*}
generates $\Rep_{\poly}(G)$ as a $\lambda$-ring, i.e., $\chi({\bf t})$, $\lambda^{2}(\chi({\bf t})),\ldots,\lambda^{n}(\chi({\bf t}))$ generate the ring $\Rep_{\poly}(G)$. (By the following identity \eqref{eqnew111}, $\chi \in \Rep_{\poly}(G)$.)
\end{lemma}

\begin{proof}
(a):  Since 
\begin{equation} \Rep (G)= \Rep (T)^W, \,\,\, \text{for any connected reductive group}\,\,G,
\end{equation}
(cf. [BD, Proposition 2.1, Chapter VI]), by Proposition \ref{prop5}, it is easy to see that
\begin{equation} \label{eqnew111} \Rep_{\poly} (G)= \Z_{\sym}[t_1^2+t_1^{-2}, \dots, t_n^2+t_n^{-2}], \,\,\,\text{for}\,\, G=\SO_{2n+1}\,\,
\text{or}\,\, \Sp_{2n},
\end{equation}
where $ \Z_{\sym}[t_1^2+t_1^{-2}, \dots, t_n^2+t_n^{-2}]$ denotes the subring of the polynomial ring $ \Z[t_1^2+t_1^{-2}, \dots, t_n^2+t_n^{-2}]$
consisting of 
symmetric polynomials. Further, the ring  $ \Z_{\sym}[t_1^2+t_1^{-2}, \dots, t_n^2+t_n^{-2}]$ is generated by the elementary symmetric polynomials 
$e_1, \dots, e_n$ in the variables $x_1= t_1^2+t_1^{-2}, \dots, x_n=t_n^2+t_n^{-2}$ (cf. [F, $\S$6.2]). In particular, to prove that $\Rep_{\poly}(G)$
is a $\lambda$-subring of $\Rep (G)$, by the axioms (3) and (5) of the definition of special $\lambda$-rings as in Definition \ref{defi6}, it suffices to prove that $\lambda^d(e_k)\in \Rep_{\poly}(G)$, for any $d \geq 1$ and $1\leq k\leq n$. Now,
\begin{equation} \label{eqn112}\lambda (t_i^2+t_i^{-2})= 1+  (t_i^2+t_i^{-2}) q+ q^2, \,\,\, \text{for any}\,\, 1\leq i\leq n.
\end{equation}
Thus, for any  $1\leq k\leq n$, using the axioms  (3) and (5) of  Definition \ref{defi6}, 
\begin{align*}\lambda (e_k)&=\prod_{1 \leq i_1 < \dots <i_k\leq n}\,\lambda (x_{i_1}\dots x_{i_k})\\
&=\prod_{1 \leq i_1 < \dots <i_k\leq n}\, \left(\sum_{d\geq 0}\,P_d(x_{i_1}, \dots , x_{i_k}) q^d\right), \,\,\,\text{by the equation}\,\eqref{eqn112},
\end{align*}
for certain symmetric polynomials $P_d$ in $k$ variables  with integral coefficients.
From this it is easy to see that the coefficient of any $q^d$ in $\lambda (e_k)$ belongs to  $ \Z_{\sym}[t_1^2+t_1^{-2}, \dots, t_n^2+t_n^{-2}]$. This proves the (a)-part from the identity \eqref{eqnew111}.

(b):  By the identity \eqref{eqn112},
\begin{align*}
\lambda\left(\sum\limits^{n}_{i=1}(t^{2}_{i}+t^{-2}_{i})\right) &= \prod\limits^{n}_{i=1}\left(1+(t^{2}_{i}+t^{-2}_{i})q+q^{2}\right)\\
&= \sum\limits^{2n}_{k=0}\left(e_{k}(t^{2}_{1}+t^{-2}_{1},\ldots,t^{2}_{n}+t^{-2}_{n})+f_{k}(t^{2}_{1}+t^{-2}_{1},\ldots,t^{2}_{n}+t^{-2}_{n})\right)q^{k},
\end{align*}
where $f_{k}$ is a symmetric polynomial with integral coefficients of degree $<k$. From this the (b)  part of the lemma follows by using the identity \eqref{eqnew111}.
\end{proof}
\begin{lemma} \label{lemmaso} For  $\SO_{2n}$, $\prod\limits^{n}_{i=1}(t_{i}-t^{-1}_{i})$ is the character of the virtual representation 
$$[V(2\omega_n)]-[V(2\omega_{n-1})],$$
where $\omega_i$ is the $i$-th fundamental representation of $\Spin_{2n}$.
\end{lemma}
\begin{proof} We first recall (cf. [BD, $\S$5.5, Chapter  VI]) that 
$$\Lambda^n(\C^{2n})=V' \oplus V^{''},\,\,\,\text{where}\,\,V'= V(2\omega_n), V^{''}=V(2\omega_{n-1})$$
as representations of $\SO_{2n}$. Recall from [Bo, Planche IV] that 
$$2\omega_n = t_1\dots t_n\,\,\,\text{and}\,\, 2\omega_{n-1} = t_1\dots t_{n-1}t_n^{-1}.$$
Moreover, the diagram automorphism of  $\SO_{2n}$ gives rise to a linear isomorphism between $V'$ and  
$V^{''}$. Under this isomorphism any weight space of $V'$ of weight  $t_1^{d_1}\dots t_{n-1}^{d_{n-1}}t_n^{d_n}$ corresponds to the weight space of $V^{''}$ of weight  $t_1^{d_1}\dots t_{n-1}^{d_{n-1}}t_n^{-d_n}$. Any weight vector of $\Lambda^n(\C^{2n})$ clearly has weight ${\bf t}^{\bf d}:=t_1^{d_1}\dots t_n^{d_n}$, for $d_i\in \{0, \pm 1\}$.   By using the invariance of the dimension of weight spaces under $W$-action, we see that
$$\dim \left( [V']_{{\bf t}^{\bf d}}\right)=\dim \left( [V^{''}]_{{\bf t}^{\bf d}}\right), \,\,\,\text{if at least one of}\,\, d_i=0,$$
where $[V']_{{\bf t}^{\bf d}}$ denotes the weight space of $V'$ of weight ${\bf t}^{\bf d}$.
So, such weights  ${\bf t}^{\bf d}$, with at least one $d_i=0$,  do not contribute to the character of $[V(2\omega_n)]-[V(2\omega_{n-1})]$.
Let us now consider the weight spaces of $V'$ and $V^{''}$ of weight ${\bf t}^{\bf d}$ where none of $d_i$ is zero. In this case, define
$${\bf d}_-:=\sharp \{1\leq i\leq n: d_i=-1\}.$$ 
Again,  using the invariance of the dimension of weight spaces under $W$-action, we see that
\begin{equation} \label{eqn301}\dim [V']_{{\bf t}^{\bf d}}=1 \,\,\,\text{if}\,\, {\bf d}_-\,\,\,\text{is even},
\end{equation}
and zero otherwise. 
 Similarly,
\begin{equation} \label{eqn302}\dim [V^{''}]_{{\bf t}^{\bf d}}=1 \,\,\,\text{if}\,\, {\bf d}_-\,\,\,\text{is odd},
\end{equation}
and zero otherwise. Combining the equations \eqref{eqn301} and \eqref{eqn302}, we get the lemma.
\end{proof}
\begin{remark}\label{rem9}
(a) It is easy to see that for the case of $G=\SO_{2n}$, $\Rep_{\poly}(G):= \Rep^\C_{\poly}(G) \cap \Rep (G)$ is {\em not} a $\lambda$-subring of
 $\Rep (G)$, by considering the function
$
\prod\limits^{n}_{i=1}(t_{i}-t^{-1}_{i})\in \Rep_{\poly}(G).
$

(b) Observe that $\chi(\mathfrak{t})=\sum\limits^{n}_{i=1}(t^{2}_{i}+t^{-2}_{i})$ is the character of the following virtual representation:
\begin{itemize}
\item[(1)] $G=\Sp_{2n}: [S^2 V]-[\Lambda^2 V]$, where $V$ is the standard representation of $\Sp_{2n}$ in $V=\mathbb{C}^{2n}$.

\item[(2)] $G=\SO_{2n+1}: [S^2 V] - [\Lambda^2 V] -[\epsilon]$,
where $V$ is the standard representation of $\SO_{2n+1}$ in $V=\mathbb{C}^{2n+1}$ and $\epsilon$ is the one dimensional trivial representation.
\end{itemize}
\end{remark}
\begin{lemma} Let $G=\Sp_{2n} (n\geq 2)$, $\SO_{2n+1} (n\geq 2) $ or $\SO_{2n} (n\geq 4)$. Then, no non-trivial irreducible representation $V(\lambda)$ (i.e., $\lambda \neq 0$) belongs to $\Rep^\C_{\poly}(G)$.
\end{lemma}
\begin{proof} We first prove the lemma for $G=\Sp_{2n}$. 	Write $\lambda= \sum_{i=1}^n d_i\omega_i$, where $\omega_i$ are the fundamental 
weights. The highest weight of  $V(\lambda)$ is given by $t_1^{d_1+\dots +d_n}t_2^{d_2+\dots +d_n}\dots
t_n^{d_n}.$ By Proposition \ref{prop5}, for $[V(\lambda)]$ to lie in $\Rep^\C_{\poly}(G)$,  $[V(\lambda)]$ must, in particular,  lie in  $\C[t_1^{\pm 2}, \dots, t_n^{\pm 2}]$. In particular, each $d_i$ should be even. Choose $i_o$ such that $d_{i_o}\neq 0$. If $i_o< n$, then  the exponent of $t_{i_o}$
in the  weight  of $f_{i_o}\cdot v_+$ (which is a nonzero vector) is odd, where $v_+$ is a highest weight vector of $V(\lambda)$ and  $f_{i_o}$ is a root vector of negative weight $-\alpha_{i_o}$ . So, assume that 
$\lambda = 2d\omega_n$. In this case, $f_n^{2d}\cdot v_+\neq 0$ and $f_n^{2d+1}\cdot v_+ = 0$. Thus, there exists an $i<n$ such that $f_if_n^{2d}\cdot v_+\neq 0$.
Again,  the exponent of  $t_{i}$ in the weight of $f_if_n^{2d}\cdot v_+$ is odd. This proves the lemma for  $G=\Sp_{2n}$.

We next consider the case of $G=\SO_{2n+1}$. In this case any irreducible module has highest weight 
$\lambda=\sum_{i=1}^n d_i\omega_i$, where $d_n$ is even.  The highest weight of  $V(\lambda)$ is given by $t_1^{d_1+\dots +d_{n-1}+\frac{d_n}{2}}t_2^{d_2+\dots +d_{n-1}+\frac{d_n}{2}}\dots
t_n^{\frac{d_n}{2}}.$  Again, by Proposition \ref{prop5}, for $[V(\lambda)]$ to lie in $\Rep^\C_{\poly}(G)$,  $[V(\lambda)]$ must, in particular,  lie in  $\C[t_1^{\pm 2}, \dots, t_n^{\pm 2}]$. For  the exponent of each $t_i$ in the highest weight of $V(\lambda)$  to be even, we must have each $d_i$ to be even and $d_n$ must be divisible by $4$.  Assuming this restriction on $\lambda$, choose $i_o$ such that $d_{i_o}\neq 0$. Then,  the exponent of  $t_{i_o}$ in the  weight of  the nonzero vector $f_{i_o}\cdot v_+$  is odd. This proves the lemma for  $G=\SO_{2n+1}.$

We finally consider the case of $G=\SO_{2n}$. In this case any irreducible module has highest weight 
$\lambda=\sum_{i=1}^n d_i\omega_i$, where $d_{n-1}+d_n$ is even, say equal to $2m$.  The highest weight of  $V(\lambda)$ is given by $t_1^{d_1+\dots +d_{n-2}+m}t_2^{d_2+\dots +d_{n-2}+m}\dots t_{n-2}^{d_{n-2}+m}
t_{n-1}^{m} t_n^{\frac{d_n-d_{n-1}}{2}}.$  By Proposition \ref{prop5}, for $[V(\lambda)]$ to lie in
 $\Rep^\C_{\poly}(G)$,  $[V(\lambda)]$ must, in particular,  lie in  
$$\C[t_1^{\pm 2}, \dots, t_n^{\pm 2}] \oplus (t_1\dots t_n)\C[t_1^{\pm 2}, \dots, t_n^{\pm 2}].$$
 In particular, 
 each of $d_1, \dots, d_{n}$ is even.
  Assuming this restriction on $\lambda$, choose $i_o$ such that $d_{i_o}\neq 0$. Then,  the weight of  the nonzero vector  $f_{i_o}\cdot v_+$  has 
exactly two exponents either odd or exactly two exponents even. This proves the lemma for  $G=\SO_{2n}$ (for $n \geq 4$).
\end{proof}
\section{Specialization of Theorem \ref{thmmain} to $G=\Sp_{2n}$ and $P$ any maximal parabolic}

We follow the notation from Section 6 and take $n \geq 2$. 

 Let $V=\mathbb{C}^{2n}$ be equipped with the
nondegenerate symplectic form $\langle \,,\,\rangle$ so that its matrix
$\bigl(\langle e_i,e_j\rangle\bigr)_{1\leq i,j \leq 2n}$ in the
standard basis $\{e_1,\dots, e_{2n}\}$ is given by the matrix $E_C$ of Section 6.
For  $1\leq r \leq n$, we let $\IG(r,2n)$ to be
the set of $r$-dimensional isotropic
subspaces of $V$ with respect to the form $\langle\,,\,\rangle$, i.e.,
$$\IG(r,2n):=\{M\in \Gr(r,2n): \langle v,v'\rangle=0,\ \forall\,  v,v'\in M\}.$$
We take $B_C:=B\cap \Sp_{2n}$ as the Borel subgroup of $ \Sp_{2n}$, where $B$ is the standard Borel subgroup of $\SL_{2n}$ consisting of upper triangular matrices of determinant $1$.
 Then, $\IG(r,2n)$  is the quotient $\Sp_{2n}/P_r^C$ of $\Sp_{2n}$ by
 the standard maximal parabolic subgroup
$P_r^C$  with $\Delta \setminus \{\alpha_r\}$ as the set of simple roots of its Levi component
$L_r^C$. (Again we take $L_r^C$ to be the unique Levi subgroup of $P_r^C$
 containing $T_C$.) 
Then, 
$$L_r^C\simeq \GL_r\times \Sp_{2(n-r)}.$$
In this case,  by the identity \eqref{eqn2} and Proposition \ref{prop5} (a), 
 $$\Rep^\mathbb{C}_{\poly}(L_r^C) \simeq \C_{\sym}[(\frac{t_1-t_1^{-1}}{2}),  \dots ,   (\frac{t_r-t_r^{-1}}{2})]\otimes_\C \C_{\sym}[(\frac{t_{r+1}-t_{r+1}^{-1}}{2})^2,  \dots ,   (\frac{t_n-t_n^{-1}}{2})^2],$$
where $ \C_{\sym}$ denotes the subalgebra of the polynomial ring consisting of symmetric polynomials. Further, by Lemma \ref{lem8}, the identity \eqref{eqnew111}  and Remark \ref{rem9} (b), 
 $$\C_{\sym}[(\frac{t_{r+1}-t_{r+1}^{-1}}{2})^2,  \dots ,   (\frac{t_n-t_n^{-1}}{2})^2]$$ is generated (as a $\C$-algebra) by the virtual representations:
$$\{\lambda^d\left([S^2(V_{2(n-r)})] - [\Lambda^2(V_{2(n-r)})]\right)\}_{1 \leq d \leq n-r},$$
where $V_{2(n-r)} =\C^{2(n-r)}$ is the standard representation of $\Sp_{2(n-r)}$ and $\lambda$ is the $\lambda$-ring structure on $\Rep(G)$ as in Example \ref{exam7}.

\begin{proposition} \label{prop8}The map $\xi^P: \Rep^\mathbb{C}_{\poly}(L_r^C)  \to H^*(\IG(r,2n), \C)$ of Theorem \ref{thmmain}, where $P=
P_r^C$,  takes
$$(t_1-t_1^{-1})+ \dots +  (t_r-t_r^{-1})\mapsto 2\epsilon^P_{s_r},$$
and
\begin{align*}(t_{r+1}-t_{r+1}^{-1})^2+ \dots +  (t_n-t_n^{-1})^2&=[S^2(V_{2(n-r)})] - [\Lambda^2(V_{2(n-r)})] -2(n-r)[\epsilon]\\
&\mapsto 4\left( (\epsilon^P_{s_r})^2+2\left( \sum_{j=r+1}^{n-1}\, (\epsilon^P_{s_j})^2\right)+ (\epsilon^P_{s_n})^2-
2  \sum_{j=r}^{n-1}\, (\epsilon^P_{s_j}\epsilon^P_{s_{j+1}})\right),
\end{align*}
where $\{\epsilon^P_{w}\}_{w\in W^P}$ is the Schubert basis of $H^*(G/P)$ as in Section 2.
\end{proposition}
\begin{proof} For $1\leq i\leq n$, let $x_i:\frt \to \C$ be the linear map  which takes 
$$\diag(x_1, \dots, x_n, x_{-n}, \dots, -x_1)\in \frt \,\,\text{ to}\,\,
x_i$$
 (cf. the identity \eqref{eq201}). By Corollary \ref{coro3}, the homomorphism   $({\theta}_{|T})^*: \C[\mathfrak{t}] \to \C[T]$,
induced from the Springer morphism $\theta$ takes  
$$ 2\left(x_{1} +\dots + x_r\right) \mapsto (t_{1}-t_{1}^{-1})+ \dots +  (t_r-t_r^{-1}) .$$
Now, the weight $x_{1} +\dots + x_r$ corresponds to the fundamental weight $\omega_r$ (cf. [Bo, Planche III]). Hence, the first part of the proposition follows from the definition of $\xi^P$ and the identity \eqref{eqnborel}.

Similarly, under   $({\theta}_{|T})^*$, 
$$ 4\left(x_{r+1}^2 +\dots + x_n^2\right) \mapsto (t_{r+1}-t_{r+1}^{-1})^2+ \dots +  (t_n-t_n^{-1})^2 .$$
Thus, the second part  follows again from the identity
\eqref{eqnborel} and Remark \ref{rem9} (b).
\end{proof}

\section{Specialization of Theorem \ref{thmmain} to $G=\SO_{2n+1}$  and $P$ any maximal parabolic}

We follow the notation from Section 6 and take $n \geq 2$.

 Let $V'=\mathbb{C}^{2n+1}$ be equipped with the
nondegenerate symmetric form $\langle \,,\,\rangle$ so that its matrix $E_B=\
\bigl(\langle e_i,e_j\rangle\bigr)_{1\leq i,j \leq 2n+1}$ (in the standard basis
 $\{e_1,\dots, e_{2n+1}\}$) is  the one given in Section 6. Note that the associated quadratic form on $V'$ is given by
\[Q(\sum t_ie_i)= t_{n+1}^2+\sum_{i=1}^n\,t_it_{2n+2-i}.\]
For  $1\leq r \leq n$, we let $\OG(r,2n+1)$ to be
the set of $r$-dimensional isotropic
subspaces of $V'$ with respect to the quadratic form $Q$, i.e.,
$$\OG(r,2n+1):=\{M\in \Gr(r,V'): Q(v)=0,\ \forall\,  v\in M\}.$$
We take $B_B:=B\cap \SO_{2n+1}$ as the Borel subgroup of $ \SO_{2n+1}$, where $B$ is the standard Borel subgroup of $\SL_{2n+1}$ consisting of upper triangular matrices of determinant $1$.
Then, $\OG(r,2n+1)$  is the quotient $\SO(2n+1)/P_r^B$ of $\SO(2n+1)$ by
 the standard maximal parabolic subgroup
$P_r^B$  with $\Delta \setminus \{\alpha_r\}$ as the set of simple roots of its Levi component
$L_r^B$. (Again we take $L_r^B$ to be the unique Levi subgroup of $P_r^B$
 containing $T_B$.)
Then,
$$L_r^B\simeq \GL_r\times \SO_{2(n-r)+1}.$$
In this case,  by the identity \eqref{eqn2} and Proposition \ref{prop5} (a), 
 $$\Rep^\mathbb{C}_{\poly}(L_r^B) \simeq \C_{\sym}[(\frac{t_1-t_1^{-1}}{2}),  \dots ,   (\frac{t_r-t_r^{-1}}{2})]\otimes_\C \C_{\sym}[(\frac{t_{r+1}-t_{r+1}^{-1}}{2})^2,  \dots ,   (\frac{t_n-t_n^{-1}}{2})^2].$$
Further, by Lemma \ref{lem8}, the identity  \eqref{eqnew111} and Remark \ref{rem9} (b), $\C_{\sym}[(\frac{t_{r+1}-t_{r+1}^{-1}}{2})^2,  \dots ,   (\frac{t_n-t_n^{-1}}{2})^2]$ is generated (as a $\C$-algebra) by the virtual representations:
$$\{\lambda^d\left([S^2(V'_{2(n-r)+1})] - [\Lambda^2(V'_{2(n-r)+1})] -[\epsilon]\right)\}_{1 \leq d \leq n-r},$$
where $V'_{2(n-r)+1} =\C^{2(n-r)+1}$ is the standard representation of $\SO_{2(n-r)+1}$ and $\lambda$ is the special $\lambda$-ring structure on $\Rep(G)$ as in Example \ref{exam7}.
\begin{proposition} \label{prop9}The map $\xi^P: \Rep^\mathbb{C}_{\poly}(L_r^B)  \to H^*(\OG(r,2n+1), \C)$ of Theorem \ref{thmmain}, where $P=P_r^B$,  takes
$$(t_1-t_1^{-1})+ \dots +  (t_r-t_r^{-1})\mapsto 2\epsilon^P_{s_r}, \,\,\,\text{if}\,\, r<n,$$
$$(t_1-t_1^{-1})+ \dots +  (t_r-t_r^{-1})\mapsto  4\epsilon^P_{s_r}, \,\,\,\text{if}\,\, r=n,$$
and
\begin{align*}(t_{r+1}-t_{r+1}^{-1})^2+ \dots + & (t_n-t_n^{-1})^2=[S^2(V'_{2(n-r)+1})] - [\Lambda^2(V'_{2(n-r)+1})] -(2(n-r)+1)[\epsilon]\\
&\mapsto 4\left ((\epsilon^P_{s_r})^2+2\left(\sum_{j=r+1}^{n-1}\, (\epsilon^P_{s_j})^2\right)+ 4 (\epsilon^P_{s_n})^2-
2 \left( \sum_{j=r}^{n-2}\, (\epsilon^P_{s_j}\epsilon^P_{s_{j+1}})\right)- 4
 (\epsilon^P_{s_{n-1}}\epsilon^P_{s_{n}})\right).
\end{align*}
\end{proposition}
\begin{proof} The proof is the same as that of Proposition \ref{prop8}.
\end{proof}

\section{Specialization of Theorem \ref{thmmain} to $G=\SO_{2n}$  and $P$ any maximal parabolic}

We follow the notation from Section 6 and take $n \geq 4$.

(a) For  $1\leq r \leq n-2$, we let $\OG(r,2n)$ to be
the set of $r$-dimensional isotropic
subspaces of $V=\C^{2n}$ with respect to the quadratic form $Q=\sum_{i=1}^n t_it_{2n+1-i}$, i.e.,
$$\OG(r,2n):=\{M\in \Gr(r,V): Q(v)=0,\ \forall\,  v\in M\}.$$
We take $B_D:=B\cap \SO_{2n}$ as the Borel subgroup of $ \SO_{2n}$, where $B$ is the standard Borel subgroup of $\SL_{2n}$ consisting of upper triangular matrices of determinant $1$.
 Then, $\OG(r,2n)$ is the quotient $\SO(2n)/P_r^D$ of $\SO(2n)$ by
 the standard maximal parabolic subgroup
$P_r^D$  with $\Delta \setminus \{\alpha_r\}$ as the set of simple roots of its Levi component
$L_r^D$. (Again we take $L_r^D$ to be the unique Levi subgroup of $P_r^D$
 containing $T_D$.)
Then,
$$L_r^D\simeq \GL_r\times \SO_{2(n-r)}.$$
In this case,  by the identity \eqref{eqn2} and Proposition \ref{prop5} (b), 
 \begin{align*}\Rep^\mathbb{C}_{\poly}  (L_r^D)  \simeq \C_{\sym}[(\frac{t_1-t_1^{-1}}{2}),  \dots ,   (\frac{t_r-t_r^{-1}}{2})]\otimes_{\C} 
 \left(R
 \oplus \left((\frac{t_{r+1}-t_{r+1}^{-1}}{2})\dots (\frac{t_{n}-t_{n}^{-1}}{2})\right) R \right) ,
\end{align*}
where $R:=\C_{\sym}[(\frac{t_{r+1}-t_{r+1}^{-1}}{2})^2,  \dots ,   (\frac{t_n-t_n^{-1}}{2})^2]$.
Further,  as for the case of $G=\Sp_{2n}$ in Section 6, $R$ is generated (as a $\C$-algebra) by the virtual representations:
$$\{\lambda^d\left([S^2(V_{2(n-r)})] - [\Lambda^2(V_{2(n-r)})] \right)\}_{1 \leq d \leq n-r},$$
where $V_{2(n-r)}=\C^{2(n-r)}$ is the standard representation of $\SO_{2(n-r)}$. Moreover, by Lemma \ref{lemmaso}, 
$ \left((t_{r+1}-t_{r+1}^{-1})\dots (t_{n}-t_{n}^{-1})\right)$ corresponds to the virtual representation $[V(2\omega_{n-r})]-
[V(2\omega_{n-r-1})]$ of $\SO_{2(n-r)}$.
\begin{proposition} \label{prop10} For $1\leq r \leq n-2$, the map $\xi^P: \Rep^\mathbb{C}_{\poly}(L_r^D)  \to H^*(\OG(r,2n), \C)$ of Theorem \ref{thmmain}, where 
$P=P_r^D$,  takes
$$(t_1-t_1^{-1})+ \dots +  (t_r-t_r^{-1})\mapsto 2\epsilon^P_{s_r},$$
\begin{align*}(t_{r+1}-t_{r+1}^{-1})^2+ \dots + & (t_n-t_n^{-1})^2=[S^2(V_{2(n-r)})] - [\Lambda^2(V_{2(n-r)})] -2(n-r)
[\epsilon]\\
&\mapsto 4\left( (\epsilon^P_{s_r})^2+2 \left(\sum_{j=r+1}^{n}\, (\epsilon^P_{s_j})^2\right)-
2 \left( \sum_{j=r}^{n-2}\, (\epsilon^P_{s_j}\epsilon^P_{s_{j+1}})\right) - 2
 \epsilon^P_{s_{n-2}}\epsilon^P_{s_{n}}\right),
\end{align*}
and
$$ \left((t_{r+1}-t_{r+1}^{-1})\dots (t_{n}-t_{n}^{-1})\right)\mapsto 2^{n-r} (\epsilon^P_{s_{r+1}}-\epsilon^P_{s_{r}})\dots 
 (\epsilon^P_{s_{n-2}}-\epsilon^P_{s_{n-3}}) (\epsilon^P_{s_{n}}+\epsilon^P_{s_{n-1}}-\epsilon^P_{s_{n-2}}) (\epsilon^P_{s_{n}}-\epsilon^P_{s_{n-1}}).$$
\end{proposition}
\begin{proof} The proof is  similar to that of   Proposition \ref{prop8} and hence omitted.
\end{proof}

(b) For $r=n$,  let $P^D_n$ be the maximal standard parabolic subgroup with $\Delta \setminus \{\alpha_n\}$ as the set of simple roots of its Levi component
$L_n^D$. Then, the partial flag variety $\SO_{2n}/P^D_n$ can be identified with  a connected component 
$\OG(n,2n)_+$ of the set of $n$-dimensional isotropic
subspaces of $V$. Moreover, the Levi subgroup 
$$L_n^D \simeq \GL_n.$$ In this case, by the identity \eqref{eqn2}, 
$$\Rep^\mathbb{C}_{\poly}  (L_n^D) \simeq \C_{\sym}[(\frac{t_1-t_1^{-1}}{2}),  \dots ,   (\frac{t_n-t_n^{-1}}{2})].$$
 \begin{proposition} The map $\xi^{P_n^D}: \Rep^\mathbb{C}_{\poly}(L_n^D)  \to H^*(\OG(n,2n)_+, \C)$ of Theorem \ref{thmmain} takes
$$(t_1-t_1^{-1})+ \dots +  (t_n-t_n^{-1})\mapsto 4\epsilon^{P^D_n}_{s_n}.$$
\end{proposition}
\begin{proof} The proof is the same as that of Proposition \ref{prop8}.
\end{proof}
Of course, the case of $r=n-1$ is parallel to the above case of $r=n$ due to the diagram automorphism taking the $n$-th node of $D_n$ to the $(n-1)$-th node.

\section{Specialization of Theorem \ref{thmmain} to the classical groups and $P=B$}

In this section $G$ is any of $\Sp_{2n}, \SO_{2n+1}$ or $\SO_{2n}$. We get the following lemma immediately from Corollary \ref{coro3}.
\begin{lemma} Under the coordinates \eqref{eqnnew201} and \eqref{eqnnew202} on the maximal torus $T$ of  $\Sp_{2n}, \SO_{2n+1}$ or $\SO_{2n}$,
$$\Rep_{\poly}^\C(T)= \C [ (\frac{t_1-t_1^{-1} } {2}), \dots, (\frac{t_n-t_n^{-1} } {2})], $$
where $T$ is thought of as the Levi subgroup of $B$.
\end{lemma}
\begin{proposition}\label{lastpropo} Under the homomorphism $\xi^B:\Rep_{\poly}^\C(T) \to H^*(G/B, \C)$ of Theorem \ref{thmmain}, 

(a) $G=\Sp_{2n}:$ 
$$ t_i-t_i^{-1} \mapsto 2 (\epsilon^B_{s_i}- \epsilon^B_{s_{i-1}}), \,\,\,\text{for any}\,\, 1\leq i\leq n .$$

(b)  $G=\SO_{2n+1}:$ 
$$ t_i-t_i^{-1} \mapsto 2 (\epsilon^B_{s_i}- \epsilon^B_{s_{i-1}}), \,\,\,\text{for any}\,\, 1\leq i < n ,$$
and 
$$ t_n-t_n^{-1} \mapsto 2 (2\epsilon^B_{s_n}- \epsilon^B_{s_{n-1}}).$$

(c)  $G=\SO_{2n}:$ 
$$ t_i-t_i^{-1} \mapsto 2 (\epsilon^B_{s_i}- \epsilon^B_{s_{i-1}}), \,\,\,\text{for any}\,\, 1\leq i \leq n-2 ,$$
$$ t_{n-1}-t_{n-1}^{-1} \mapsto 2 (\epsilon^B_{s_{n-1}} + \epsilon^B_{s_{n}}- \epsilon^B_{s_{n-2}}),$$
and
$$ t_{n}-t_{n}^{-1} \mapsto 2 (\epsilon^B_{s_{n}} - \epsilon^B_{s_{n-1}}).$$
\end{proposition}
\begin{proof} The proof is similar to the proof of Proposition \ref{prop8} and hence omitted.
\end{proof}
\bibliographystyle{plain}
\def\noopsort#1{}

\vskip5ex

\noindent
Address: Shrawan Kumar,
Department of Mathematics,
University of North Carolina,
Chapel Hill, NC  27599--3250. 
\noindent
email: shrawan@email.unc.edu

\end{document}